\documentclass[12pt]{article}
\usepackage{amsthm,enumerate}
\usepackage{amssymb}
\usepackage{amsfonts, mathrsfs}
\usepackage{amsmath}
\usepackage{algorithm}
\usepackage{hyperref}
\usepackage{color}\usepackage{graphicx}
\allowdisplaybreaks

\newcommand{\pr}{\mathbf{Pr}}

\newcommand{\e}{{\mathbf E}}

\newcommand{\set}[1]{\left\{ #1 \right\}}

 \newcommand{\bag}{\begin{align}}
\newcommand{\bags}{\begin{align*}}
\newcommand{\eag}{\end{align*}}
\newcommand{\eags}{\end{align*}}

\newtheorem{thm}{Theorem}
\newtheorem{lem}[thm]{Lemma}

\newtheorem{cor}[thm]{Corollary}
\newtheorem{dfn}[thm]{Definition}

\newtheorem{rem}[thm]{Remark}

\newcommand\cA{\mathcal A}

\newcommand\cC{\mathcal C}
\newcommand\cD{\mathcal D}

\newcommand\cP{\mathcal P}

\newcommand{\brac}[1]{\left(#1\right)}
\newcommand{\bfrac}[2]{\brac{\frac{#1}{#2}}}
\def\nono{\nonumber}
\def\t{\tau}
\def\e{\epsilon}
\def\b{\beta}
\def\m{\mu}
\begin{document}
\title{Pattern Colored Hamilton Cycles in Random Graphs}
\author{Michael Anastos\thanks{Department of Mathematical Sciences, Carnegie Mellon University, Pittsburgh PA15213. Research supported in part by NSF Grant DMS1661063. Email: manastos@andrew.cmu.edu.} \and Alan Frieze\thanks{Department of Mathematical Sciences, Carnegie Mellon University, Pittsburgh PA15213. Research supported in part by NSF Grant DMS1661063: Email alan@random.math.cmu.edu.}}

\maketitle
\begin{abstract}
We consider the existence of patterned Hamilton cycles in randomly colored random graphs. Given a string $\Pi$ over a set of colors $\set{1,2,\ldots,r}$, we say that a Hamilton cycle is $\Pi$-colored if the pattern repeats at intervals of length $|\Pi|$ as we go around the cycle. We prove a hitting time for the existence of such a cycle. We also prove a hitting time result for a related notion of $\Pi$-connected. 
\end{abstract}
\section{Introduction}
In recent years there has been a growing interest in the properties of randomly colored random graphs. The edge-colored random graph process can be desribed as follows: let $G_0,G_1,...,G_{N}$, $N=\binom{n}{2}$, be the random graph process. That is $e_i=E(G_{i})\setminus E(G_{i-1})$ is chosen uniformly at random from all the edges not present in $E(G_{i-1})$. For $i\in [N]$, at step $i$, a random color $c_i$  is chosen independently and uniformly at random from $[r]$ and is assigned to $e_i$. We denote this randomly $[r]$-colored version of the random graph process by $G_0^r,G_1^r$,...,$G_N^r$. 

Much of the interest in this model has been focussed on {\em rainbow} colorings. A set $S$ of edges is said to be rainbow colored if every edge of $S$ has a different color. One of the earliest papers on this subject is due to Frieze and Mckay \cite{FrMcK}. In this paper, $r=cn$ where $c\geq 1$ is a constant. Let $\t_0=\min\set{i:G_i\text{ is connected}}$ be the hitting time for connectivity and let $\t_1=\min\set{i:\text{at least $n-1$ distinct colors have been used}}$. Then \cite{FrMcK} shows that $\t_0=\t_1$ w.h.p. 

After this, the attention has been focussed on the question of when does there exist a rainbow Hamilton cycle. Cooper and Frieze \cite{CoFr} showed that $O(n\log n)$ random edges and $O(n)$ colors are sufficient. This was improved to $(1+o(1))n\log n$ random edges and $(1+o(1))n$ colors by Frieze and Loh \cite{FrLo}. This was sharpened still further by Ferber and Krivelevich \cite{FeKr} who showed that the number of edges can be reduced to the exact threshold for Hamiltonicity.  Bal and Frieze considered the case where exactly $n$ colors are available and showed that $O(n\log n)$ random edges are sufficient to obtain a rainbow Hamilton cycle.

The next phase of this study, concerns Hamilton cycles in random $k$-uniform hypergraphs. There are various notions of Hamilton cycle in this context, and Ferber and Krivelevich \cite{FeKr} proved that if $m_H$ edges are needed for a given type of hamilton cycle to exist w.h.p. then $O(m_H)$ random edges and $(1+\e)m_H$ colors are sufficient for a rainbow Hamilton cycles. The hypergraph results in \cite{FeKr} were sharpened by Dudek, English and Frieze \cite{DEF}.

Cooper and Frieze \cite{CoFr1} considered the related question of what is the threshold for every $k$-bounded coloring of the edges of $G_{n,m}$ to contain at least one rainbow Hamilton cycle. Here $k$-bounded means that no color can be used more than $k$ times. 

Rainbow is one pattern of coloring and Espig, Frieze and Krivelich \cite{EFK} considered other colorings. Suppose that $r$ is constant. They considered the existence of hamilton cycles where the edges of the cycle are colored in sequence $1,2,\ldots,r,1,2,\ldots,r,\ldots$. When $r=2$ they called such colorings {\em Zebraic}. They gave tight results in terms of the number of random edges needed for such cycles. Our first result generalises this and considers arbitrary patterns of coloring.

An \emph{$[r]$-pattern} $\Pi$ is a finite sequence with elements in $[r]$. For a given  $[r]$-pattern $\Pi$ let $\ell=|\Pi|$ be its length and for $1\leq j\leq \ell$ let $\Pi_j$ be its $j$'th element. We say a path/cycle $f_1,f_2,...,f_k$ is $\Pi$-colored if there exists an integer $0\leq l\leq \ell-1$ such that $\forall j\in[k]$, $f_j$ is colored by $\Pi_{j+l}$. 
So for example if $[r]=3$, $\Pi=1,2,2,3$ and $P=f_1,f_2,...,f_6$ is  a $\Pi$-colored path then $f_1,f_2,...,f_6$ may have colors $1,2,2,3,1,2$ or $2,2,3,1,2,2$ or $2,3,1,2,2,3$ or $3,1,2,2,3,1$ respectively.
Here and elsewhere for a given pattern $\Pi$ and $s>|\Pi|$ we let $\Pi_s=\Pi_j$ where $j=s\mod \ell$. In this paper we are interested in the following question. Given $r=O(1)$ and an $r$-pattern $\Pi$ when does the first $\Pi$-colored Hamilton cycle appears in the process $G_0^r,G_1^r$,\ldots,$G_N^r$.   

We will assume without loss of generality that $[r]=\set{\Pi_1,\Pi_2,\ldots,\Pi_\ell}$ i.e. every color in $[r]$ appears at least once in $\Pi$. We will also assume that $\ell$ divides $n$.

Let $r\in \mathbb{N}$ and let $\Pi$ be an $[r]$-pattern. Then,
$$ \tau_{\Pi}:=\min \{i: G_i^r \text{ contains a $\Pi$-colored Hamilton cycle} \}$$
For $v\in V$ we say that $v$ {\em fits} the pattern $\Pi$ in $G_i^r$ if there exist two distinct edges $e_1,e_2\in E(G)$ incident to $v$ and $1\leq j\leq\ell$  such that  $e_1$ has color $\Pi_j$ and $e_2$ has color $\Pi_{j+1}$. So in our example where $[r]=3$, $\Pi=1,2,2,3$ $v$ fits  $\Pi$ if it is incident to some $e_1,e_2$ with colors $c_1,c_2$ that satisfy $\{c_1,c_2\}\in \set{\set{1,2},\set{2,2},\set{2,3},\set{3,1} }.$
Furthermore we define the hitting time
$$ \tau_{fit-\Pi}:=\min \{i: \text{every } v \in V \text{ fits } \Pi \text{ in } G_i^r\}.$$
Observe that before $\tau_{fit-\Pi}$ occurs there is at least one vertex $v$ that ``does not fit the pattern".
That is there do not exist two colors that appear in adjacent places in the pattern and in the neighborhood of $v$. Thus clearly for any pattern $\Pi$ we have $\tau_{fit-\Pi}\leq \tau_{\Pi}$. 
\begin{thm}\label{Thm1}
Let $r=O(1)$ and $\Pi$ be an $[r]$-pattern. Then, w.h.p.\@
 $\tau_{fit-\Pi}=\tau_{\Pi}.$
\end{thm}
Our second result is related to the notion of {\em rainbow connection}. Given a connected graph $G$, the rainbow connection $r_c(G)$ is defined as the smallest $r$ such that there exists an $r$-coloring of the edges of $G$ so that there is a rainbow path between every pair of vertices of $G$. This is a well studied concept, see Li, Shi and Sun \cite{lisun} for a survey. The rainbow connection of random graphs has also been studied. Heckel and Riordan \cite{HR} and He and Liang \cite{heliang} studied the rainbow connection of dense random graphs. Frieze and Tsourakakis \cite{hnprainbow} studied the rainbow connection of a random graph at the connectivity threshold. Dudek, Frieze and Tsourakakis \cite{DFT} and Kamce\v{v}, Krivelevich and Sudakov \cite{KKS} and Molloy \cite{Moll} studied the rainbow connection of random regular graphs. Suffice it to say that in general the rainbow connection is close to the diameter in all cases. Espig, Frieze and Krivelevich \cite{EFK} introduced the notion of {\em zebraic connection}. Given a 2-coloring of the edges of a connected graph $G$ we say that a path is {\em zebraic} if the colors of edges alternate along the path. A colored graph is zebraicly connected if there is a zebraic path joining every pair of vertices. In the paper \cite{EFK}, they proved a hittng time result for the zebraic connectivity of a random 2-coloring of the edges of a random graph. In this paper we generalise this notion to $\Pi$-connectivity. We say that $G$ is $\Pi$-connected if every two vertices are joined by a $\Pi$-colored path.  Also when does $G_i^r$ become $\Pi$-connected. For both questions we give a hitting time result.
In order to state the result about $\Pi$-connectivity we define the following hitting times: 
$$\tau_{1}:=\min\{i: G_i\text{ has minimum degree at least }1\}, $$

$$ \tau_{\Pi-connected}:=\min\{i: G_i^r \text{ is $\Pi$-connected}\}.$$
\begin{thm}\label{Thm2}
Let $r=O(1)$ and $\Pi$ be a non trivial $[r]$-pattern. Then w.h.p.\@
$$\tau_{\Pi-connected}= \tau_{1-\Pi}.$$
\end{thm}
The following corollary is then immediate from our knowledge of $\tau_{1-\Pi}$ .
\begin{cor}
Let $m=\frac12n(\log n+c_n)$. Then,
$$\lim_{\substack{n\to\infty\\\ell|n}}\Pr( G_i^r \text{ is $\Pi$-connected})=\begin{cases}0&c_n\to -\infty.\\e^{-e^{-c}}&c_n\to c.\\1&c_n\to +\infty.\end{cases}
$$
\end{cor}
\subsection{Directed versions}
There some natural directed versions of the results that we just have just stated. 
For that we consider the directed random graph process $D_0,D_1,\ldots,D_{N'}$, $N'=n(n-1)$. Here
$e_i=E(D_{i})\setminus E(D_{i-1})$ is chosen at random from all the  $n(n-1)-(i-1)$ arcs not present in $E(D_{i-1})$.
For $i\in [N']$, at step $i$, a random color $c_i$  is chosen independently and uniformly at random from $[r]$ and is assigned to $e_i$. We denote this randomly colored version of the directed random graph process by $D_0^r,D_1^r$,\ldots,$D_{N'}^r$.

The notion of $\Pi$ paths/cycle/connectivity can be extended in a straightforward manner to the directed setting by substituting directed path/ directed cycle in the place of cycle/paths. For $v\in V$ we say that $v$ fits the pattern $\Pi$ if there exist arcs
$e_1,e_2\in E(D)$  and $1\leq j<\ell$  such that 
 $e_1$ has color $\Pi_i$, $e_2$ has color $\Pi_{i+1}$ and such that the head of $e_1$ and the tail of $e_2$ are both $v$. 
We replace connectivity in the undirected setting by strong connectivity in the directed case.
Finally let $\overrightarrow{ \tau}_{\Pi}$, $\overrightarrow{ \tau}_{fit-\Pi}$, $\overrightarrow{ \tau}_{\Pi-connected}$ and $\overrightarrow{ \tau}_{1}$ 
be the directed analogs of $\tau_{\Pi}$, $\tau_{fit-\Pi}$, $\tau_{\Pi-connected}$ and $\tau_{1}$.
\begin{thm}\label{Thm3}
Let $r=O(1)$ and $\Pi$ be an $[r]$-pattern. Then, w.h.p.\@
 $$\overrightarrow{\tau}_{fit-\Pi}=\overrightarrow{\tau}_{\Pi}.$$
\end{thm}

\begin{thm}\label{Thm4}
Let $r=O(1)$ and $\Pi$ be an $[r]$-pattern. Then w.h.p.\@
$$\overrightarrow{\tau}_{\Pi-connected}= \overrightarrow{\tau}_{1}.$$
\end{thm}
\begin{cor}
Let $m=n(\log n+c)$. Then,
$$\lim_{\substack{n\to\infty\\\ell|n}}\Pr( G_i^r \text{ is $\Pi$-connected})=\begin{cases}0&c_n\to -\infty.\\ e^{-2e^{-c}}   &c_n\to c.\\1&c_n\to +\infty.\end{cases}
$$
\end{cor}
\subsection{Notation-Preliminaries}\label{1.2}
For $i\in [N]$ and $c\in [r]$ let $G_i^r(c)$ denote the subgraph of $G_i^r$ induced by the edges of color $c$. Furthermore denote by ${\text deg}_{i}(v,c)$ the degree of $v$ in $G_i^r(c)$. 
By extension, for $C\subset [r]$ denote by $G_i^r(C)$  the subgraph of $G_i^r$ induced by the edges with color in $C$ and set ${\text deg}_{i}(v,C)$ to be the degree of $v$ in $G_i^r(C)$.
For $e\in E(G_N)$ let $c(e)$ be the color that is assigned to $e$ by the end of the process.
Throughout the paper we will use the following estimate.
\begin{lem}\label{est}
Let $a,b,c,d,i,t,Q\in \mathbb{Z}_{\geq 0}$ be such that $b,c,i,t=o(Q)$, $d,i=o(t)$ and $i=o(b)$ then,
$$\frac{\binom{a}{i}\binom{Q-c-b}{t-d-i}}{ \binom{Q-c}{t-d} } \leq \bigg(\frac{3at}{iQ}\bigg)^i
\exp\set{-\frac{(1+o(1))bt}{Q}}. $$
\end{lem}
\begin{proof}
\begin{align*}
\frac{\binom{a}{i}\binom{Q-c-b}{t-d-i}}{\binom{Q-c}{t-d}} 
&=\frac{\binom{a}{i}\binom{Q-c-b}{t-d-i} \binom{t-d}{i}}{\binom{Q-c}{i}\binom{Q-c-i}{t-d-i} }
= \binom{t-d}{i}\prod_{j=0}^{i-1}\frac{a-j}{Q-c-j} \prod_{h=0}^{t-d-i-1}\frac{Q-c-b-h}{Q-c-i-h}\\
&\leq \bigg(\frac{et}{i} \bigg)^i \cdot \prod_{j=0}^{i-1}\frac{a}{Q-c} \prod_{h=0}^{t-d-i-1} \brac{1-\frac{b}{Q-c-i-h}}\\
&\leq \bigg(\frac{et}{i} \bigg)^i \bigg(\frac{(1+o(1))a}{Q} \bigg)^i \bigg(1- \frac{(1+o(1))b}{Q}\bigg)^{(1+o(1))t}\\
&\leq  \bigg(\frac{(1+o(1))eat}{iQ} \bigg)^i \exp\set{-\frac{(1+o(1))bt}{Q}}.
\end{align*}
\end{proof}
We will also use the following elementary result.
\begin{lem}\label{binomial}
Let $\nu=\nu(n)$ be a positive integer and let $0<p<1$ be such that $\nu p\to \infty$ and let $X$ be a $Binomial(\nu,p)$ random variable. Then w.h.p.\@ $X=(1+o(1))\nu p$.
\end{lem}
\begin{rem}\label{gnm}
Let $C\subset [r]$ and $t\geq n$. Then $G_t^r(C)$ is distributed as $G(n,m)$ i.e a random graph with $m$ edges chosen  at random from all $N$ edges. Here $m$ is the number of edges in $G_t^r(C)$ colored by a color in $C$. Since each edge is colored independently and at random, $m$ is distributed as a $Binomial(t,|C|/r)$ random variable. Therefore w.h.p.\@ $m=(1+o(1))|C|t/r$.
\end{rem}
The paper contains various constants that are used throughout. We collect them here for ease of reference:
\begin{itemize}
\item $r$ equals the number of colors available.
\item $\ell$ equals the length of pattern $\Pi$.
\item $\e=10^{-9}\ell^{-1}$.
\item $t_i=i\m$ where $\m=\e n\log n$ for $1\leq i\leq 2\ell$.
\item $X_i=\set{e_j:\;j\in [(i-1)\m+1,i\m]}$ and $Y_i=\set{e_j:\;j\in [(\ell+i-1)\m+1,(\ell+i)\m]}$ for $i\in [\ell]$.
\item $V_i=\set{{\frac{(i-1)n}{\ell}+1},{\frac{(i-1)n}{\ell}+2},\ldots,{\frac{in}{\ell}}}$ for $1\leq i\leq \ell$, a partition of the vertex set into $\ell$ equal size subsets.
\item $n_\ell=\frac{n}{\ell}$ equals the size of the $V_i$'s. We use this notation to stop formulae looking too ugly.
\item $\b=10^{-3}\ell^{-1}r^{-1} \epsilon$.
\item $BAD$ is a set of low degree vertices with size bounded by $2\ell n_b$ where $n_b=n^{1-10\b}$.
\item $n_r=2\ell|BAD|$.
\end{itemize}
\section{Demand of a pattern}
For a given  $[r]$-pattern $\Pi$ the hitting time $\tau_{fit-\Pi}$ clearly depends on 
both the number of colors $r$ and the pattern its self. In order to determine the later we define the {\em demand} of a pattern. 
\begin{dfn}
Let 
$$\cD(\Pi):=\set{S\subset [\ell]:  \set{i,i+1} \cap  S \neq \emptyset \text{ for all } i\in [\ell]}. \hspace{10mm}(\ell+1=1\text{ here})   $$ 
\end{dfn}
\begin{dfn}
Let $r\in \mathbb{N}$ and let $\Pi$ be an $[r]$-pattern. The ``demand'' of $\Pi$ is 
$$ d(\Pi):= \min\set{|\set{\Pi_i:i\in S}|: S\in \cD(\Pi)}.$$
\end{dfn} 
The motive for giving the above definition is the following.
For a given  $[r]$-pattern $\Pi$,
if there exists a set $S\in \cD(\Pi)$ and a vertex $v\in V$ such that $v$ is not incident with any edge colored with one of the at least  $d(\Pi)$ colors in $\set{\Pi_i:i\in S}$ then $v$ does not fit $\Pi$. Conversely if $v$ does not fit $\Pi$ then for $i\in[\ell]$, $v$ is incident to at most 1 edge of color $P_i$ or to at most 1 edge of color $P_{i+1}$. Here we say at most one instead of none since we have to consider the case $\Pi_i=\Pi_{i+1}$. Therefore $S=[\ell]\backslash \set{\text{$i$:\;$v$ is incident with two edges of color $i$}}$ satisfies $|\{\Pi_i:i\in S\} |\geq d(\Pi)$ and certifies that $v$ does not fit $\Pi$.

\vspace{3mm}
In the following lemma we use the following well know result (see \cite{KarFr}). Let $\epsilon>0$ then, w.h.p. 
$G\big(n,\frac{(1-\epsilon)n\log n}{2}\big)$ contains an isolated vertex. On the other hand  $G\big(n,\frac{(1+\epsilon)n\log n}{2}\big)$ does not contain a vertex of degree at most $r$+1.
\begin{lem}\label{boundssub}
Let $r=O(1)$, $\Pi$ be an $[r]$-pattern and $\epsilon>0$.
 Then, w.h.p.\@  $$ \frac{r}{d(\Pi)}\cdot\frac{(1-\epsilon) n\log n}{2}  \leq \tau_{fit-\Pi} \leq   \frac{r}{d(\Pi)}\cdot  \frac{(1+\epsilon)n\log n}{2}.$$ 
\end{lem}
\begin{proof}
Let $t = \frac{r}{d(\Pi)}\cdot \frac{(1-\epsilon)n\log n}{2}$. Let $S\in \cD(\Pi)$ be such that $|\set{\Pi_i:i\in S}|=d(\Pi)$
and let $A=\set{\Pi_i:i\in S}$. $G_t^r(A)$, is distributed as $G(n,m)$ where w.h.p.\@  $m\leq  \frac{(1-\epsilon/2)n\log n}{2} $   (see Remark \ref{gnm}). Hence w.h.p.\@ $G_t^r(A)$ has  an isolated vertex. This vertex is not incident to any edge with a color in  $A$ and it does not fit $\Pi$. Consequently,  w.h.p. \@ $ \frac{r}{d(\Pi)}\cdot\frac{(1-\epsilon) n\log n}{2}  \leq \tau_{fit-\Pi}.$
\vspace{3mm}
\\Now let  $t = \frac{r}{d(\Pi)}\cdot \frac{(1+\epsilon)n\log n}{2}$. In the event that  $ \tau_{fit-\Pi}\geq t$ we have that there is a vertex $v\in V$ and a set $S \in \cD(\Pi)$ such that $S$ certifies that $v$ does not fit $\Pi$ i.e. for every $i\in \set{\Pi_i:i\in S}$ $v$ is incident to at most one edge with color $i$. Hence if we let 
$A=\set{\Pi_i:i\in S}$ then $v$ has degree at most $r$ in $G_t(A)$.
Fix $S \in \cD(\Pi)$ and let $A=\set{\Pi_i:i\in S}$.  Then $|A|\geq d(\Pi)$. Furthermore $G_t^r(A)$, is distributed as $G(n,m')$ where w.h.p.\@ 
 $m\geq  \frac{(1+\epsilon/2)n\log n}{2}$  (see Remark \ref{gnm}). Hence w.h.p.\@ $G_t^r(A)$  has no  vertex of degree at most r. Then by taking a union bound over  $S \in \cD(\Pi)$ we get that w.h.p.\@ $\tau_{fit-\Pi} \leq   \frac{r}{d(\Pi)} \cdot \frac{(1+\epsilon)n\log n}{2}$.
\end{proof}
\begin{cor}\label{crude}
Let $r=O(1)$, $\Pi$ be an $[r]$-pattern. Then w.h.p.\@  $ 0.4 n\log n   \leq \tau_{fit-\Pi} \leq   r n\log n.$ 
\end{cor}
\begin{proof}
Follows from Lemma \ref{boundssub} and the fact that $1\leq d(\Pi)\leq r$.
\end{proof}
The above corollary can be tightened. In actual fact, Theorem \ref{Thm1} implies the following: let $d=d(\Pi)$ and $D=|\cD(\Pi)|$.
\begin{cor}\label{corx}
Let 
$$m=\frac{rn}{2d}(\log n+c_n).$$ 
Then,
$$\lim_{\substack{n\to\infty\\\ell|n}}\Pr( G_i^r \text{ contains a $\Pi$-colored Hamilton cycle})=\begin{cases}0&c_n\to -\infty.\\e^{-\lambda}&c_n\to c.\\1&c_n\to +\infty.\end{cases}
$$
Here $\lambda$ depends on the pattern $\Pi$ and it is equal to the expected number of vertices that do not fit $\Pi$.
\end{cor}
The justification for Corollary \ref{corx} comes from Theorem \ref{Thm1} and the fact that in the case $c_n\to c$, the number of vertices that do not fit $\Pi$ is asymptotically Poisson with mean $\lambda$. The proof of the corollary follows a standard ``method of moments'' proof and is omitted. 
\section{Outline proof of Theorem \ref{Thm1}} 
We have defined a partition of $V$ into sets $V_1,...,V_\ell$ of equal size (see Section \ref{1.2}). We begin by identifying sets of {\em bad} vertices that have low degree and then show that (i) there are few of them and (ii) they are spread out in $G_{\tau_{fit-\Pi}}$. This is the content of Lemma \ref{str}.  Next, we cover the vertices in $BAD$ by a set of $\Pi$-colored paths $\cP_{bad}$ with a good endpoint in each of $V_1,V_\ell$ such that each bad vertex lies in the interior of such a path. For this we use Algorithm $CoverBAD$ and prove that it is successful in Lemma \ref{cover}.

If $\cA$ is the set of vertices not covered by $\cP_{bad}$ then the sets $V_i\cap\cA$ may be unbalanced. We move a small set of vertices around so that $\cA$ is now partitioned into equal sized sets $V_i''$. Then for each $i\in [\ell-1]$ we find a perfect matching of color $\Pi_i$ from $V_i''$ to $V_{i+1}''$. These matchings together form a collection of $\Pi$-colored paths $\cP_{good}$ that cover the vertices in $\cA$, each with an endpoint in $V_1,V_\ell$. The edges used in the construction of these matchings are all in $\bigcup_{i=1}^{\ell-1}X_i\cup \bigcup_{i=1}^{\ell-1}Y_i$ (the sets $X_i,Y_i$ are defined in Section \ref{1.2}). Let $\cP=\cP_{bad}\cup \cP_{good}=\set{P_1,P_2,\ldots,P_{n_\ell}}$. Let the endpoints of $P_i$ be $v_i^-\in V_1$ and $v_i^+\in V_\ell$ for $1\leq i\leq n_\ell$.

After this, we find a perfect matching $M=\set{(v_i^+,v_{\pi(i)}^-):i\in [n_\ell]}$ of color $\Pi_\ell$ from $V_\ell$ to $V_1$, using a subset of the edges $X_\ell\cup Y_\ell$. Here $\pi$ is a permutation of $[n_\ell]$ and so the digraph $\Gamma=([n_\ell],\set{(i,\pi(i))})$ is a collection of vertex disjoint cycles. We argue by symmetry that $\pi$ is a random permutation so that w.h.p. it has at most $2\log n$ cycles. A cycle $i,\pi(i),\pi^2(i),\ldots,i$ can be expanded into a $\Pi$-colored cycle $v_i^-,P_i,v_i^+,v_{\pi(i)}^-,P_{\pi(i)},v_{\pi(i)}^+,v_{\pi^2(i)}^-,\ldots,v_i^-$. And in this way we cover the vertex set $[n]$ by $O(\log n)$ $\Pi$-colored cycles.

After this we focus on converting this set of cycles into a single $\Pi$-colored Hamilton cycle using $\Pi_\ell$ edges of $E(G_{\tau_{fit-\Pi}})\setminus \bigcup_{i=1}^\ell (X_i\cup Y_i)$. This turns out to be essentially equal to the task successfully faced in the construction of a directed Hamilton cycle in \cite{HamFrieze}, and which is laid out more explicitly in \cite{2016Anastos}. The reduction to \cite{HamFrieze}, \cite{2016Anastos} is laid out in Section \ref{reduction}.
\section{Structural results}
For every $i\in [N]$ an ordering of the endpoints of $e_i$ is chosen independently and uniformly at random. Hence we may consider that $e_i$ is given to us in the form of an ordered pair $\vec{e}_i=(v_i,w_i)$. Note that in the proof of Theorem \ref{Thm1}, we know that we will not be presented with both of $(v,w)$ and $(w,v)$.
\begin{dfn}
For $j\in [\ell]$ define the sets
\begin{align*}
BAD_j&=\set{v:|\{i:\;v_i=v,w_i\in V_j, t_{j-1}< i\leq t_j, c_i=\Pi_j \}|\leq \b\log n}.\\
BAD_{j+\ell}&=\set{v:|\{i: w_i=v, v_i\in V_j, t_{j+\ell-1}< i\leq t_{j+\ell}, c_i=\Pi_j \}|\leq \b\log n}.
\end{align*}
Set $BAD=\underset{j\in [2\ell]}{\bigcup}BAD_j$ and call every vertex in $BAD$ bad. Also set $GOOD=V\setminus BAD$ and call every vertex in $GOOD$ good.
\end{dfn}
\begin{dfn}
We also define the vertex set 
\begin{align*}
TBAD&=\set{v \in V: \exists C\subset [r]  \text{ such that } 
|C|=d(\Pi) \text{ and  } {\text deg}_{\tau_{fit-\Pi}}(v,c)\leq \log\log n\text{ for all }c\in C}\\
&\subseteq BAD.
\end{align*}
We say that a vertex in $TBAD$ is tbad (terribly bad). 
\end{dfn}
\begin{lem}\label{str}
The following hold w.h.p.\@
\begin{description}
\item[(a)] $|BAD|\leq 2\ell n^{1-10\b}$.
\item[(b)] Every vertex has at most $ 6 \epsilon^{-1}r \ell^2 $ bad vertices within distance $2\ell$ of it in $G_{\tau_{fit-\Pi}}$.
\item[(c)]  $\nexists v,w\in V$ s.t.\@ $v\in TBAD$, $w\in BAD$ and their distance in $G_{\tau_{fit-\Pi}}$ is less than $2\ell$.
\item[(d)] The maximum degree in $G_{\tau_{fit-\Pi}}$ is less than $10r\log n$.
\end{description}
\end{lem}
\begin{proof}(a) Let $n_b= n^{1-10\b}$. Due to symmetry, for $j\in [2\ell]$ the sizes of $BAD_j$ follow the same distribution. Therefore it suffices to show that with probability  1- $o(1)$ we have $|BAD_1|\leq n_b.$ In the case that $|BAD_1|>n_b$ there is a set $A\subset V$ of size $n_b$ such that 
\begin{align}\label{con1} 
|\set{\vec{e}_i=(v_i,w_i):v_i\in A, w_i\in V_1\setminus A, c_i=\Pi_1 \text{ and } i\leq t_1}|\leq \b n_b\log n.
\end{align}
$G_{t_1}^r(\Pi_1)$ is distributed as a $G(n,t_1')$ where w.h.p.\@ $t_1'=(1+o(1))t_1$ (see Remark \ref{gnm}). 
We can choose $A$ in $\binom{n}{n_b}$ ways. Then there are at least $n_b(n_\ell-n_b)$ and at most $n_bn_\ell$ edges with one endpoint in each of $A,V_1\setminus A$. From these, if \eqref{con1} occurs, then $k\leq \beta n_b\log n$ many appear in  $G_{t_1}^r(\Pi_1)$.
Therefore, 
\begin{align*}
\pr(\eqref{con1}\mid t_1')&\leq \binom{n}{n_b} \sum_{k=0}^{\b n_b\log n} \frac{\binom{n_bn_{\ell}}{k}\binom{N-n_b(n_{\ell}-n_b)}{t_1'-k} }{\binom{N}{t_1'}}\\
&\leq \bigg(\frac{en}{n_b} \bigg)^{n_b}   \sum_{k=0}^{\b n_b\log n} \bigg( \frac{3t_1' n_b n_\ell}{k N}\bigg)^{k} \exp\set{-\frac{(1+o(1)) n_b n_\ell t_1'}{N}}\\
&\leq \bigg(\frac{en}{n_b} \bigg)^{n_b}   \sum_{k=0}^{\b n_b\log n} \bigg( \frac{6e^{o(1)} \epsilon n_b \log n }{k\ell r}\bigg)^{k} 
\exp\set{-\frac{2e^{o(1)}\epsilon n_b\log n }{r\ell}}\\
&\leq \b n_b\brac{e^{1+10\b}\cdot (6e^{o(1)}\cdot 10^3\cdot e^{-2e^{o(1)}\cdot10^3} )^{\b\log n}}^{n_b}\\
&=o(1).
\end{align*}
For the second inequality we use Lemma \ref{est} (with $a=n_bn_\ell,b=n_b(n_\ell-n_b),c=d=0,i=k,t=t_1',Q=N$) and for the third one we use that $t_1'=(1+o(1)) \epsilon r^{-1} \log n$. For the last one we use that $\big( \frac{ 6 e^{o(1)} \epsilon n_b \log n }{k\ell r}\big)^{k}$ has a unique maximum obtained when $\frac{6e^{o(1)} \epsilon n_b \log n }{k\ell r}=e$. Thus for $k\in [0,\b n_b\log n]$, this is maximized when $k=\b n_b\log n$.
\vspace{3mm}
\\(b) We will show that with probability  at least 1- $o\big(1)$ every vertex has at most 
$s=3 \epsilon^{-1} r \ell $ vertices within distance $2\ell$ of it that belong to $BAD_1$ in $G_{rn\log n}$, hence in $G_{\tau_{fit-\Pi}}$. The result follows by symmetry and the union bound. In the case that there exists a vertex $v$ with at least $s$ vertices in $BAD_1$ within distance $2\ell$ of $v$ we can find sets $A,B$ such that the following is satisfied: (i) $|A|=s$, (ii) $|B|\leq 2s(\ell-1)+1$ (iii) $A \cup B$ spans a tree in $G_{rn\log n}$, (iv) there are at most $\b s\log n$ edges from $A$ to $V_1\backslash (A \cup B)$ in $G^r_{t_1}(\Pi_1)$ (here $A$ consists of vertices in $BAD_1$ and $B$ consists of a vertex $v$ and at most $(2\ell-1)s$ vertices that are spanned by paths of length at most $2\ell$ form $v$ to vertices in $A$).
\vspace{3mm}
\\Fix sets $A,B$ satisfying (i), (ii) and a tree $T$ that is spanned by $A\cup B$. Let $|B|=b$.
\begin{align}
\pr(E(T)\subset E(G^r_{rn\log n})) &=\frac{\binom{N}{rn\log n-(s+b-1)}}{ \binom{N}{rn\log n}}
= \frac{ \binom{rn\log n}{s+b-1}}{ \binom{N-rn\log n+(s+b-1)}{s+b-1}}\nonumber\\
&\leq \bigg( \frac{rn\log n}{N-rn\log n+(s+b-1)}\bigg)^{s+b-1}\nonumber\\
&\leq \bigg(\frac{3r\log n}{n} \bigg)^{s+b-1}.\label{Tree}
\end{align}
$G^r_{t_1}(\Pi_1)$ is distributed as a $G(n,t_1')$ with $t_1'=(1+o(1))t_1/r$ (see Remark \ref{gnm}).
Conditioned on $t_1'$ and $E(T)\subset E(G^r_{rn\log n})$  and on $z\leq s+b-1$ edges  of $T$  appearing in $G^r_{t_1}(\Pi_1)$
we have that the probability of condition (iv) being satisfied is bounded above by  
\begin{align*}
\sum_{k=0}^{\b s\log n}\frac{ \binom{ s n_{\ell} }{k} \binom{N-s(n_{\ell}-s-b)-(s+b-1)}{t_1'-k-z}}{ \binom{N-(s+b-1)}{t_1'-z}} 
&\leq \sum_{k=0}^{\b s\log n}\bigg(\frac{6e^{o(1)}snt_1'}{\ell k n^2} \bigg)^{k}  \exp\set{-\frac{(1+o(1))sn t_1'}{\ell N}} \\
&=\sum_{k=0}^{\b s\log n}\bigg(\frac{6e^{o(1)}s \epsilon \log n}{\ell k r} \bigg)^{k}  \exp\set{-\frac{2e^{o(1)}s \epsilon \log n }{\ell r}}
\\
&\leq \b s\log n(6e^{o(1)}\cdot 10^3\cdot e^{-2e^{o(1)}\cdot10^3})^{\b s\log n}\\
&\leq \frac{1}{n^2}.
\end{align*}
To get the first expression observe that $G^r_{t_1}(\Pi_1)$ consists of $t_1'$ edges. z of those have already being chosen from $E(T)$. Thereafter $k$ of those are chosen so that they have an endpoint in each $A,V_1\setminus (A\cup B)$ and the rest are chosen from those not in $E(T)$ or those having both an endpoint in each $A, V_1\setminus(A\cup B)$. We then apply Lemma \ref{est} with $a=sn_\ell,b=s(n_\ell-s-b),c=s+b-1,d=z,i=k,t=t_1',Q=N$. For the last equality we have used that $(\frac{6e^{o(1)}s \epsilon \log n}{k})^k$ is maximized when $\frac{6e^{o(1)}s \epsilon \log n}{k}=e$. Thus for $k\in[0,n_{c_2}]$ this is maximized when $k=n_{c_2}$.
\vspace{3mm}
\\Summarizing, there are $\binom{n}{s}$ ways to choose $A$ and thereafter $\binom{n-s}{b}$ ways to choose a set $B$ of size $b\leq (2\ell-1)s+1$. Given $A,B$ there are $(s+b)^{s+b-2}$ trees that are spanned by $A\cup B$. Each such tree appears with probability at most $\big(\frac{3r\log n}{n} \big)^{s+b-1}$. Finally given the appearance of any such tree there are at most $\beta s \log n$ edges of color $\Pi_1$  with an endpoint in each $A$, $V_1\backslash (A\cup B)$ in $G_{t_1}$ with probability at most $1/n^2$
Therefore the probability that conditions (i), (ii), (iii) and (iv) are satisfied is bounded by 
\begin{align*}
\sum_{b=0}^{(2\ell-1)  s+1}\binom{n}{s}\binom{n-s}{b}(s+b)^{s+b-2}\bigg(\frac{3r\log n}{n} \bigg)^{s+b-1} \frac{1}{n^2} =o(1).
\end{align*}
(c) It is enough to show the above statement for $v\in TBAD$, $w\in BAD_1$. 
In the case that the statement is false $\exists v,w\in V$, $C\subset [r]$ with $|C|=d$ and $S\subset V$ with $|S|=s \leq 2\ell-1$ such that the following hold:
(i) $\set{v,w}\cup S$ spans a path $P$ in $G_{r\log n}$, 
(ii) there are at most $d\log \log n$ edges adjacent to $v$ in $G_{\tau_{fit-\Pi}}^r(C)$,
(iii) there are at most $\b\log n$ edges from $w$ to $V_1\setminus S\cup \set{w}$ in $G_{t_1}^r(\Pi_1)$. 
\vspace{3mm}
\\Fix such $v,w,P,S,C$. As shown in (b), (see \eqref{Tree}), (i) is satisfied with probability at most $(\frac{3r\log n}{n})^{s+1}$. $G_{\tau_{fit-\Pi}}^r(C)$ is distributed as a $G(n,m_d)$ where w.h.p.\@ $\frac12(1-\alpha)n\log n\leq m_d\leq rn\log n$ for arbitrarily small $\alpha>0$, (see Corollary \ref{crude} and Remark \ref{gnm}). Therefore conditional on (i) and on $u\leq s+1$ edges of $P$ appearing in $G_{\tau_{fit-\Pi}}^r(C)$ out of which at most one is adjacent to $v$ we have 
\begin{align*}
\pr({\text deg}_{\tau_{fit-\Pi}}(v,C)\leq d\log\log n)&\leq \sum_{k=0}^{d\log\log n-1}\frac{\binom{n-2}{k}\binom{N-(s+1)-(n-2)}{m_d-u-k}}{\binom{N-(s+1)}{m_d-u}}\\
&\leq \sum_{k=0}^{d\log\log n} \bigg(\frac{6e^{o(1)}nm_d}{kn^2} \bigg)^k \exp\set{ -\frac{(1+o(1))nm_d}{N} }\\
&\leq  \sum_{k=0}^{d\log\log n} \bigg(\frac{6e^{o(1)}r\log n}{k} \bigg)^k  e^{-(1-2\alpha)\log n}  \\
&\leq n^{3\alpha-1}.
\end{align*}
For the first inequality we used Lemma \ref{est} with $a=b=n-2,c=s+1,d=u,i=k,t=m_d,Q=N$.

$G_{t_1}^r(\Pi_1)$ is distributed as a $G(n,t_1')$ where w.h.p.\@ $t_1=(1+o(1))t_1/r$ (see Remark \ref{gnm}. Conditional on (i), (ii) occurring, if (iii) also occurs, then in $E(G_{t_1}(\Pi_1))$ there are $k\leq \beta \log n$ edges from $w$ to $V_1\setminus \set{v,w}$ and  $h\leq s+1+d\log\log n$ edges that either belong to $E(P)$ or are adjacent to $v$ and lie in $E(G_{\tau_{fit-\Pi}}^r(C\cap \set{\Pi_1}))\cap E( G_{t_1'}^r(\Pi_1))$. The remainder of the $t_1-k-h$ edges are chosen from those not in $E(P)$ and not in $\{w\}\times V_1\setminus \set{v,w}$. If $\Pi_1\in C$ then these edges are also chosen from those not incident to $v$ while if $\Pi_1\notin C$ then these edges are also chosen from the  $d_{\tau_{fit-\Pi}}(v,C)$ edges not incident to $v$ in $E(G_{\tau_{fit-\Pi}}^r(C))$. Let $j=n-2$ if $\Pi_1\in C$ and $j=d_{\tau_{fit-\Pi}}(v,C)$ otherwise. Then,
\begin{align*}
\pr((iii)|(i),(ii))&\leq\sum_{k=0}^{\b\log n}\frac{\binom{n_{\ell}}{k} \binom{N-(s+1)-(n_{\ell}-2)-(j-1)}{t_1'-h-k}}{ \binom{N-(s+1)-j}{t_1'-h}} \\
&\leq \sum_{k=0}^{\b\log n} \bigg(\frac{6e^{o(1)}t_1'n}{k\ell n^2} \bigg)^k\exp\set{-\frac{(1+o(1))t_1'n}{\ell N}}\\
&\leq  \sum_{k=0}^{\b\log n} \bigg(6e^{o(1)}\cdot 10^3 \bigg)^{\b\log n}\exp\set{-\frac{3\epsilon \log n}{2r \ell }}\\
& \leq n^{-\e/(r\ell)}.
\end{align*}
For the first inequality we used Lemma \ref{est} with $a=n_\ell,b=n_\ell-2,c=j+s+1,d=h,i=k,t=t_1',Q=N$.
For the second one we used the fact that $ (\frac{6e^{o(1)}t_1'n}{k\ell n^2} )^k$ has a unique global maximum that occurs when $\frac{6e^{o(1)}t_1'n}{k\ell n^2}=e$. Thus for $k\in [0,\b\log n]$ this is maximized when $k=\b\log n$. Furthermore we used that w.h.p. $ \frac{3t_1}{4r}\leq t_1' \leq \frac{7t_1}{6r}$ (see Remark \ref{gnm}).

Taking the union bound over $v,w,P,S,C$ we have that the probability the event described by (c) does not hold can be bounded by
$$ n(n-1)\sum_{s=0}^{2\ell-1} \binom{n-2}{s} s! \binom{r}{d} \bigg(\frac{3r\log n}{n}\bigg)^{s+1} n^{3\alpha-1}\cdot  n^{-\e/(r\ell)}=o(1),$$
for $\alpha$ sufficiently small.

(d) In the event that (d) fails there exist a vertex with at least $10r\log n$ in $G_{rn \log n}$. Thus
\begin{align*}
\pr(\neg(d))= \sum_{k=10r\log n}^{n} n \frac{\binom{n-1}{k}\binom{N-n-1}{rn\log n-k}}{\binom{N}{rn\log n}}
\leq n \sum_{k=10r\log n}^{n^2} \bigg(\frac{6e^{o(1)}n \cdot rn \log n}{kn^2}\bigg)^{k}=o(1).
\end{align*}
For the first inequality we used Lemma \ref{est} with  $a=n-1,c=d=0,e=n+1,i=k,t=rn\log n,Q=N$.
\end{proof}
\section{Proof of Theorem \ref{Thm1}}
We say a pattern is trivial if $\ell$=1. In the case of a trivial pattern, Theorems \ref{Thm1}-\ref{Thm4} reduce to well known results about hitting times and the random graph process (see \cite{EvolutionGnm}, \cite{KarFr}). Hence we may assume that $\ell>1$.
\vspace{3mm}
\\We prove Theorem \ref{Thm1} using the 3-phase approach for finding a Hamilton cycle in the directed random graph process used in \cite{HamFrieze}. In the first phase a 1-factor is created consisting of $O(\log n)$ cycles. Then in the second phase we sequentially merge pairs of cycles by performing two arc exchanges until no such arc exchange is available. W.h.p.\@ at the end of Phase 2 the largest cycle is of order $n-o(n)$. Finally as a last phase we merge one by one the smaller cycles with the largest one. In each merging we start by merging the two cycles into a path. Then we use double rotations, a technique that generalizes Posa's rotations to the directed setting, in order to turn the initial path into a cycle. Instead of going through all 3 phases of the proof we give a reduction to the following Lemma. For its proof see Sections 9 and 10 of \cite{2016Anastos}.
\begin{lem}\label{reduction}
Let $G',F,H,G_1$ be di-graphs such that: (i) $F$ is a 1-factor consisting of $O(\log n)$ directed cycles, (ii) $H$ is edge disjoint from $F$ and its maximum total degree is $O(\log n)$, (iii) $G_1$ is a random graph where every edge not in $E(F)\cup E(H)$ appears independently with probability $p_1=\Omega(\frac{\log n}{n})$, (iv) $E(F)\cup E(H)\cup E(G_1)\subset E(G')$. Then w.h.p.\@ $G'$ has a Hamilton cycle consisting of arcs only in $E(F)\cup E(G_1)$.
\end{lem}   
In the random digraph setting Lemma \ref{reduction} states the following. Assume that we are able to perform Phase 1 and find a directed 1-factor $F$ consisting of $O(\log n)$ cycles without exposing too many arcs. Even if we forbid re-using any the arcs that we have exposed but not used in the construction of $F$, (each vertex is w.h.p. incident to $O(\log n)$ of them), then we have enough randomness left so that w.h.p.\@ we are able to perform Phases 2,3 and construct the Hamilton cycle while avoiding the edges that we have exposed in Phase 1.  \\
\subsection{Construction of the 2-factor}
We now reveal the following. For $i\in [N]$ we reveal the color $c_i$. For every $i\leq t_{\ell}$ we reveal $v_i$ from the pair $\vec{e}_i=(v_i,w_i)$.
 Furthermore for  $t_\ell< i\leq t_{2\ell}$ we reveal $w_i$ from the pair $\vec{e}_i=(v_i,w_i)$.
Given the information that we have just revealed we can determine $BAD$.
Now given the set $BAD$ we reveal every edge with an endpoint in $BAD$. Observe that $TBAD\subset BAD$ hence $TBAD$ is now determined. We now implement the  algorithm \emph{CoverBAD}, given below, in order to cover every vertex in $BAD$ by a unique path with an endpoint in each $V_1,V_\ell$. 

Given $v\in V$, \emph{CoverBAD} grows a $\Pi$-colored path $P(v)$ with $v$ in the interior. Suppose that $v$ is incident with edges $(v_1,v_0=v),(w_0=v,w_1)$ of color $\Pi_i,\Pi_{i+1}$ respectively, see Step 2. If $i\neq\ell$ then in Step 3 we create a path $(v=w_0,w_1,w_2,\ldots,w_s,s=\ell-i-1)$ where edge $(w_j,w_{j+1})$ has color $\Pi_{i+j}$. In Step 4 we extend this path via $(v=v_0,v_1,\ldots,v_t,t=i)$ where edge $(v_j,v_{j+1})$ has color $\Pi_{i-j}$. Thus in this case $P(v)$ has length $\ell$. When $i=\ell$ we grow a path of length $2\ell$ in the same manner.

We say that \emph{CoverBAD} breaks if there is a step where no vertices satisfying the given conditions can be found.

\begin{algorithm}
\caption{\textsc{CoverBAD }}
{\bf Initialize}: $\cA:=GOOD, \Pi_{bad}:=\emptyset$.\\
For $v\in BAD$ do begin:
\begin{enumerate}
\item{ Set $v_0=w_0=v$, $s=1,t=1$.}
\item{ Find $v_0,w_1 \in \cA $, $i\in [\ell]$ such that $(v_0,v_1), (w_0,w_1)\in G_{\tau_{fit-\Pi}}$, $c((v_0,v_1))=\Pi_i$ and $c((w_0,w_1))=\Pi_{i+1}.$  Update $\cA=\cA \setminus \set{v_0,w_1}$.}
\item{While $c((w_{s-1},w_s))\neq \Pi_\ell$:
 expose all ordered edges $\vec{e}_i=(w_s,x)$ and find $x\in \cA$ such that  $c((w_s,x))=\Pi_{i+s+1}$. 
 Update$\cA=\cA \setminus\set{x}, s=s+1, w_s=x$.}  
\item{While $c((v_{t-1},v_t))\neq \Pi_1$:
 expose all ordered edges $\vec{e}_i=(y,v_t)$ and find $y\in \cA$ such that  $c((y,v_t))=\Pi_{i-(t-1)}$. 
 Update$\cA=\cA \setminus\set{y}, t=t+1$, $v_t=y.$}
\item{Set $P(v)=v_t,v_{t-1},\ldots,v_1,v,w_1,\ldots,w_s$. Update $\Pi_{bad}=\Pi_{bad}\cup \set{P(v)}$.}
\item end
\end{enumerate}
\end{algorithm}
\begin{rem}\label{distance}
Every vertex $v\in BAD$ lies in the interior of some path in $\Pi_{bad}$ and it is at distance at most $\ell$ from each of its endpoints.
\end{rem}
\begin{lem}\label{cover}
W.h.p.\@ \emph{CoverBAD} does not break. 
\end{lem}
\begin{proof}
We consider 3 cases.
\vspace{2mm}
\\{\bf Case 1:} \emph{CoverBAD} breaks at Step 2 for some $v\in TBAD$. Due the definition of the stopping time $\tau_{fit-\Pi}$, there exist $v_0,w_1\in V$, and $i\in [\ell]$ such that $(v_0,v), (v,w_1)\in G_{\tau_{fit-\Pi}}$, $c((v_0,v))=\Pi_i, c((v,w_1))=\Pi_{i+1}$. It is enough to show that at the begining of the iteration in which we construct $P(v)$, we can find $v_1,w_1 \in \cA \subseteq GOOD$. Lemma \ref{str} (iii) implies that $v_1,w_1\notin BAD$. If $v_1 \notin \cA$ then there exists $v'\in BAD$ such that $v_0 \in P(v')$. Then Remark \ref{distance} implies that $v,v'$ are within distance $2\ell$ contradicting Lemma \ref{str} (iii).
Hence w.h.p.\@ $v_1\in \cA$. Similarly $w_1 \in \cA$ w.h.p.
\vspace{2mm}
\\{\bf Case 2:}  \emph{CoverBAD} breaks at Step 2 for some $v\in BAD \setminus TBAD$. Observe that if for every $i \in [\ell]$ we have either
$ {\text deg}_{\tau_{fit-\Pi}}(v,\Pi_i)\leq \log \log n$ or  $ {\text deg}_{\tau_{fit-\Pi}}(v,\Pi_{i+1})\leq \log \log n$  then 
$|\set{c\in C: {\text deg}_{\tau_{fit-\Pi}}(v,c)\leq \log\log n}|\geq d$. On the other hand since $v\notin TBAD$ we have that 
$|\set{c\in C: {\text deg}_{\tau_{fit-\Pi}}(v,c)\leq \log\log n}|\leq d-1$. Hence there exists $j\in [\ell]$ such that
$ {\text deg}_{\tau_{fit-\Pi}}(v,\Pi_i)\geq \log \log n$ and  $ {\text deg}_{\tau_{fit-\Pi}}(v,\Pi_{i+1})\geq \log \log n$.
 Let 
$$\cC_j(v)=\set{w\in V: (v,w)\in G_{\tau_{fit-\Pi}} \text{ and } c((v,w))= \Pi_j}.$$ 
 Similarly define $\cC_{j+1}(v)$. Then $|\cC_j(v)|,|\cC_{j+1}(v)|\geq \log\log n .$ Since the algorithm breaks we have 
 $\cC_j(v)\cap\cA=\emptyset$ or  $\cC_{j+1}(v)\cap\cA=\emptyset$. From Remark \ref{distance} we have that if a vertex is removed from $\cA$ then it is within distance $\ell$ from some vertex in $BAD.$ From Lemma \ref{str} (ii)
w.h.p.\@ there are at most $6\epsilon^{-1}r\ell^2$ vertices within distance $2\ell$ from $v$ and for each such vertex at most $2\ell$ vertices are removed from the $\cA$. Hence w.h.p.\@ $|\cC_j(v)\cap\cA| \geq |\cC_j(v)|-12\epsilon^{-1}r\ell^3 \geq 1$. Thus $\cC_j(v)\cap \cA\neq \emptyset$. Similarly $\cC_{j+1}(v)\cap \cA\neq \emptyset$.
\vspace{2mm}
\\{\bf Case 3:} \emph{CoverBAD} breaks at Step 3 or Step 4 for some $v\in BAD$. Assume that it breaks at Step 3 for $v\in V$ (the case that it breaks at Step 4 can be dealt with in the same way). Then there exists $j\in [\ell]$ and $t\leq \ell-1$ such that no vertex $v_{t+1}$ can be found such that $c((v_t,v_{t+1}))$ is of color $\Pi_j$.
 Let $\cC_j(v)=\set{v\in V: (v_t,v_{t+1})\in G_{\tau_{fit-\Pi} } \text{ and }c((v_t,v_{t+1}))=\Pi_j}.$ By construction $v_s\in GOOD$ and hence $|\cC_j(v)|=\Omega(\log n)$. The rest of the argument is identical to the one given for Case 2.
\end{proof}
After the termination of \emph{CoverBAD}, $\cA$ consists of all the vertices not spanned by some path in $\Pi_{bad}$. Our next step is to cover the vertices in $\cA$ by $\Pi$-colored paths. In order to do so we use the partition $V_1$,\ldots,$V_\ell$.
For $i \in [\ell]$ let $V_i'=V_i \cap \cA$.
From each set $V_i$ a set of size at most $2\ell|BAD|$ may have been used in the construction of paths in $\Pi_{bad}$.
Thus for $i\in [\ell]$ we have $|V_i\cap \cA| \geq |V_i|-2\ell |BAD|$. 
Let $n_r=2\ell |BAD|$. Then from Lemma \ref{str} we have that $n_r=O(n^{1-10\beta})$.
To equalize the sizes of $V_i'$s,  for $i\in [\ell]$ we now remove from $V_i'$ a  random set of size $n_r$. We redistribute the vertices that we have just removed to the sets $V_i'$s in order to form sets of the same size.
For $i\in [\ell]$ we let $V_i''$ be the set resulted from $V_i'$.
\vspace{3mm}
\\ We  now define the following edge sets. 
For $v\in V$ and $i \in [\ell]$ let $E^+_i(v)$ be the first 6 edges $(v,w)$ with $w\in V_{i}'\cap V_i''$ that appear after $e_{t_{i-1}}$. Similarly let $E^-_i(v)$ be the first 6 edges $(w,v)$ with $w\in V_{i}'\cap V_{i}''$ that appear after $e_{t_{\ell+i-1}}$. We stress that $V_i'\cap V_i''$ equals the set of the vertices in $V_i$ that are not covered by some path in $\Pi_{bad}$ and have not been moved to some other set $V_j'$.  
\begin{lem}\label{edgematch}
W.h.p.\@ for every $i\in [\ell]$ and every $v\in V_i'$ we have $E^+_i(v) \subset E(G_{t_{i-1}})\setminus E(G_{t_i})$ and $E^-_i(v) \subset E(G_{t_{\ell+i}})\setminus E(G_{t_{\ell+i-1}})$.
\end{lem}
\begin{proof}
Because of symmetry it is enough to show that for a fixed $v\in V_1''$ we have
$$\pr( |\set{w\in V_2'\cap V_2'': \vec{ e_i}=(v,w) \text{ with } t_1< i \leq t_2} |<6 )=o(n^{-1}).$$
Let $N_2(v)= \set{w\in V_2': \vec{ e_i}=(v,w) \text{ with } t_1< i \leq t_2}.$ Since $v\in \cA$ we have that $v\notin BAD_2$. Furthermore there are at most $O(1)$ vertices in $BAD$ within distance $2\ell$ of $v$   hence at most $O(1)$ of its neighbors do not lie in $\cA$. Thus $n_2(v)=|N_2(v)| =\Theta(\log n)$. 
Therefore
\begin{align} 
&\pr( |\set{w\in V_2'\cap V_2'': \vec{ e_i}=(v,w) \text{ with } t_1< i \leq t_2} |\leq 5 )\nono\\
\leq& \sum_{k=0}^5 \frac{\binom{n_2(v)}{n_2(v)-k}\binom{|V_2'|-n_2(v)}{n_r-n_2(v)+k}}{ \binom{|V_2'|}{n_r}}\label{redist}\\ 
\leq& \sum_{k=0}^5n_2^k(v)
\prod_{i=0}^{n_2(v)-k-1}\frac{n_r-i}{|V_2'|-i} \prod_{i=0}^{n_r-n_2(v)+k-1}\frac{n_r-n_2(v)+k-i}{|V_2'|-n_2(v)+k-i}\cdot\frac{|V_2'|+n_2(v)-i}{n_r-n_2(v)+k-i}
\nono\\
\leq&\sum_{k=0}^5 n_2^k(v)\cdot\bfrac{n_r}{|V_2'|}^{n_2(v)-k}    \nono\\
=&o(n^{-1}).\nono
\end{align}
To see inequality \eqref{redist} observe that on the event $\set{|N_2(v)\cap V_2''|=k\leq 5} $ at least $n_2(v)-k$ of the vertices in $N_2(v)$ were chosen and redistributed. The last inequality follows from the fact that 
$n_r=O(n^{1-10\beta}), n_2(v)=\Theta(\log n)$ and $|V_2'|=(1+o(1))n_\ell$.
\end{proof}
For $i \in [\ell]$ set 
$$E_i^+=\underset{v\in V''_i}{\bigcup}E^+_{i+1}(v) \hspace{5mm} \text{ and } \hspace{5mm} 
 E_i^-=\underset{v\in V''_i}{\bigcup}E^-_{i-1}(v).$$
Thus $E^+_i$ ($E^-_i$ respectively) is a set of $6|V_i''|$ edges with an endpoint in each $V_i'', V_{i+1}'\cap V_{i+1}''$ ($V_i'', V_{i-1}'\cap V_{i-1}''$ resp.) such that each vertex in $V_i''$ is incident to 6 edges in  it.
\begin{lem}\label{match}
W.h.p.\@ for $i\in[\ell-1]$,  $E_i^+\cup E^-_{i+1}$ spans a complete matching $M_i$ from $V_i''$ to $V_{i+1}''$.
\end{lem}
\begin{proof}
Assume that no such matching exists. Then Hall's theorem implies that either (i) $\exists A\subseteq V_i'',B\subseteq V_{i+1}''$, with $|A|=s, |B|=s-1, 6\leq s\leq \frac{|V_i''|}{2}$, and no edge from $A$ to $V_{i+1}''\setminus A$ lies in  $E_i^+$ or (ii) $\exists A\subseteq V_{i+1}'',B\subseteq V_{i}''$, with $|A|=s, |B|=s-1, 6\leq s\leq \frac{|V_{i+1}''|}{2}$, and no edge from $A$ to $V_{i}''\setminus A$ lies in  $E^-_{i+1}$. In this context, given that we have shown that (i) is unlikely, we have to be sure that the edge choices involved in (ii) are independent of those considered in (i). This is achieved by the fact that the edges in $E_i^+,E_i^-$ are sampled from distinct sets. This does not mean complete independence because we cannot choose the same edge twice and this accounts for the $O(\log n)$ term in \eqref{use18}.

Therefore the probability that $\exists i\in[\ell-1]$ such that  $E_i^+\cup E^-_{i+1}$ does not span a matching $M_i$ between $V_i''$ and $V_{i+1}''$ is bounded by
\begin{align}
&2(\ell-1)\sum_{s=6}^{|V_1''|/2}\sum_{A\in \binom{V_1''}{s}}\sum_{B\in \binom{V_2''}{s+1} } \prod_{v\in B}\frac{\binom{|A\cap V'_1\cap V_1''|}{6}}{\binom{|V_1'\cap V_1''|-O(\log n)}{6}}\label{use18}\\
\leq& 2(\ell-1)\sum_{s=6}^{|V_1''|/2}
\binom{|V_1''|}{s}\binom{|V_1''|}{s+1} 
\bigg(\frac{e^{o(1)}\binom{s}{6}}{\binom{|V_1''|}{6}}\bigg)^s\label{approx}\\
\leq& 2(\ell-1)\sum_{s=6}^{|V_1''|/2}
\bigg(\frac{e|V_1''|}{s}\bigg)^s \bigg(\frac{e|V_1''|}{s}\bigg)^{s+1} \bigg(\frac{e^{o(1)}s}{|V''|}\bigg)^{6s}\nono\\
\leq& n^2 \sum_{s=6}^{|V_1''|/2}\bigg(\frac{e^{2+o(1)} s^4}{|V_1''|^4}\bigg)^s\nono \\
=&o(1).\nono
\end{align}
For \eqref{approx} we used that $|V_1''|=|V_2''|=(1-o(1))|V_1'\cap V_1''|$
and at the last equality that $|V_i''|= (1+o(1))n_{\ell}$.
\end{proof}
The edges in $\cup_{i\in [\ell-1]} M_i$ span a set of $\Pi$-colored paths with an endpoint in each $V_1''$, $V_\ell''$ that covers $\cA$. Let $\cP_M$ be this set  of paths. Set $\cP=\cP_M \cup \cP_{bad}=\set{P_1,P_2,\ldots,P_{n_h}}$. For $P_i\in \cP$, we denote its endpoint in $V_1$ by $v_i^-$ and its endpoint in $V_\ell$ by $v_i^+$. In addition set $Q^+=\set{v_i^+:P_i\in \cP}$ and $Q^-=\set{v_i^-:P_i\in \cP}$. Let $\cP_{good}'\subset \cP$ be the set of paths with an endpoint in each of $V_1'\cap V_1''$, $V_\ell'\cap V_\ell''$. We define the following edge sets. For $v\in Q^+$ let $E^+(v)$ be the first 6 edges $(v,w)$ with $w\in \set{v^-_i:P_i\in \cP_{good}}$ that appear. Similarly, for $v\in Q^-$ let $E^-(v)$ be the first 6 edges $(w,v)$ with $w\in \set{v^+_i:P_i\in \cP_{good}}$ that appear after $e_{t_{2\ell-1}}$. Note that because $Q^+,Q^-\subseteq \cA$, these edges are not conditioned by the edges of the matchings in Lemma \ref{match}. Finally set 
$$E^+=\underset{v\in Q^+}{\bigcup} E^+(v) \hspace{5mm} \text{ and } \hspace{5mm}  E^-=\underset{v\in Q^-}{\bigcup}E^-(v).$$
We have the following two Lemmas. Their  proofs  are identical to the proofs of Lemmas \ref{edgematch} and \ref{match} respectively and hence are omitted.
\begin{lem}\label{edgematch*}
 W.h.p.\@  we have $E^+  \subset E(G_{t_{1}})$ and
$E^- \subset E(G_{t_{2\ell}})\setminus E(G_{t_{2\ell}-1})$.
\end{lem}
\begin{lem}\label{match*}
W.h.p.\@   $E^+\cup E^-$ spans a complete matching $M^*$ from $Q^+$ to $Q^-$.
\end{lem}
We now use $M^*$ to join the paths in $\cP$ and create a 2-factor $F'$. Thus $E(F')=M^*\cup (\cup_{i\in [\ell-1]}M_i)$. We finish this subsection with the following Lemma.
\begin{lem}
W.h.p.\@ $F'$ consists of $O(\log n)$ cycles.
\end{lem} 
\begin{proof}
The key observation is that given $M_i,i\in[\ell-1]$, there is a one to one correspondence between realizations of $F'$ and permutations on $\set{v^+_i:P_i\in \cP_{good}}$. Here we are using the fact that our pattern is not trivial (i.e.\@ $\ell>1$). In addition due to the construction of $E^+,E^-$ every cycle in $F'$ contains a vertex in $\set{v^+_i:P_i\in \cP_{good}}$.
Furthermore, due to symmetry every matching $M^*$ that may occur, occurs equally likely and thus  each possible permutation  on $\set{v^+_i:P_i\in \cP_{good} }$ is equally likely to occur. Each cycle of $F'$ corresponds to a cycle of the permutation generated by $M^*$. It is well known that w.h.p.\@ a uniformly random permutation on $M$ elements consists w.h.p.\@ of at most $2\log M$ cycles and the lemma follows.
\end{proof}
\subsection{Reduction to Lemma \ref{reduction}}
We now generate graphs $G',H,F,G_1$ as follows that satisfy the conditions of Lemma \ref{reduction}. These graphs will have vertex set $V'$ distinct from $[n]$. $G'$ is designed so that a Hamilton cycle in $G'$ can be used to construct a $\Pi$-colored Hamilton cycle in $G^r_\tau$. There is a vertex $v(P_i)\in V'$ for each $P_i \in \cP$. We let 
$$F=\set{(v(P_i),v(P_j)): P_i,P_j\in \cP \text{ and } (v^+_i,v^-_j)\in M^*}$$ Then w.h.p. $F$ defines a collection of at most $2\log |V'|$ cycles that span $V'$. We let $H$ be the graph consisting of the arcs in $H'\setminus F$ where 
$$H':= \set{(v(P_i),v(P_j)): P_i,P_j\in \cP \text{ and } (v^+_i,v^-_j) \in E(G_{t_{2\ell}}^r)  }.$$
Let $p_1=\frac{\log n}{100r n}$, $B=\set{f_1,f_2,\ldots,f_b}$ be the edges in $\set{(v^+_i,v^-_j): P_i,P_j\in \cP}\cap (E(G_{\tau_{fit-\Pi}})\setminus E(G_{2\ell}))$ of color $\Pi_\ell$. In order to generate $G_1$ we first generate $k(F,H)$ i.i.d.\@ $Bernoulli(p_1)$ random variables $X_i$ with probability of success  $p_1= \frac{\log n}{100 r  n}$. Here $k(F,H)$ equals the number of arcs spanned by $V'$ and not included in $E(F)\cup E(H)$.
Let $N_1$ be the number of $X_i's$ with $X_1=1$. We generate $E(G_1)$ as follows: For $i \in [N_1]$ we include for each edge in $\{f_1,,\ldots,f_{N_1}\}$ the corresponding arc, that is for $f_i=(P^+_i,P^-_{i'})$ we include the arc $(v(P_i),v(P_{i'}))$.
\vspace{3mm}
\\Now let $G'=(V',E')$ where $E'=E(F)\cup E(H) \cup E(G_1)$. It follows from Lemma \ref{str}(d) that $H$ has maximum degree $O(\log |V'|)$. Since $\Pi$ is not trivial every path in $\cP$ has length at least 2 hence its endpoints are distinct.  Therefore there is a one to one correspondence between arcs of $G'$ not in $E(F)\cup E(H)$ and edges in $B'=\set{(v^+_i,v^-_j): P_i,P_j\in \cP}\setminus E(G_{t_{2\ell}})$. Moreover on the event that $N_1\leq |B|$, any set $A$ of $N_1$ arcs corresponding to edges in $B'$ is equally likely to satisfy $A=\set{f_1,f_2,\ldots,f_{N_1}}$. Therefore on the event $N_1\leq |B|$, $E(G_1)$ has the same distribution as if we conditioned on $N_1$ edges appearing in the model where every edge not in $E(F)\cup E(H)$ appears independently with probability $p_1=\Omega(\frac{\log n}{n})$. Hence in the case that $N_1\leq |B|$, it follows from Lemma \ref{reduction} that  $G'$ has a Hamilton cycle. Any such cycle $Q$ corresponds to a $\Pi$-colored  Hamilton cycle $Q'$  in $G_{\tau_{fit-\Pi}}^r$. Here $Q'$ is  the Hamilton cycle induced by the union of the edges corresponding to the arcs in $Q$ and the edges spanned by the paths in $\cP_M$. $Q'$ is $\Pi$-colored because every vertex in $V'$ corresponds to a $\Pi$-colored path in $G_{\tau_{fit-\Pi}}^r$ that starts with color $\Pi_1$, ends with color $\Pi_{\ell-1}$ and every edge in $E(F) \cup E(G_1')$ corresponds to a $\Pi_{\ell}$-colored edge in $G_{\tau_{fit-\Pi}}^r$
\begin{lem}
W.h.p.\@ $|B|\geq N_1$.
\end{lem}
\begin{proof}
We will show that w.h.p. $N_1<\frac{ n \log n}{10r\ell^2 }<|B|$. 
Let $X$ be the number of $\Pi_\ell$ colored edges in $E(G_{\tau_{fit-\Pi}})\setminus(E(G_{2\ell}))$. 
Each edge is $\Pi_\ell$ colored independently with probability $1/r$. Therefore $X$ is distributed as a $Binomial(\tau_{fit-\Pi}-t_{2\ell},1/r)$ random variable.
 Corollary \ref{boundssub} and Lemma \ref{binomial} imply that   w.h.p.\@ $X\geq \frac{3n\log n}{10 r}.$ 
The edges in $X$ are choosen uniformly at random from those not in $G_{t_{2\ell}}$. Since $|\cP|=n_{\ell}-|O(BAD)|$, Lemma \ref{str}(d) implies that w.h.p.\@ at most $10r n_\ell \log n$ out of the $(1+o(1))\frac{n^2}{\ell^2}$ edges in 
$\set{(v^+_i,v^-_j): P_i,P_j\in \cP}$ lie in $G_{t_{2\ell}}$. Therefore
\begin{align*}
\pr\bigg (|B|\leq \frac{n\log n}{10 r \ell^2 } \bigg)&\leq \sum_{k=0}^{\frac{n \log n}{10 r\ell^2 }}\frac{\binom{n^2/\ell^2}{k}\binom{N-t_{2\ell}-(1+o(1))n^2/\ell^2}{\frac{3n\log n}{10r}-k}}{ \binom{N-t_{2\ell}}{\frac{3n\log n}{10r}}}
\\ &\leq \sum_{k=0}^{\frac{n \log n}{10 r\ell^2 }} \bigg(\frac{9n^3\log n}{10r\ell^2kN} \bigg)^k\exp \set{-\frac{(1+o(1))3n^3\log n}{10r\ell^2N}}\\
&\leq \sum_{k=0}^{\frac{n \log n}{10 r\ell^2 }} \bigg(\frac{9n^3\log n}{10r\ell^2kN} \bigg)^k
\exp \set{-\frac{n\log n}{2r\ell^2}}\\
&\leq {\frac{n \log n}{10 r\ell^2 }} \cdot 20^{\frac{n \log n}{10 r\ell^2 }}
\exp \set{-\frac{n\log n}{2r\ell^2}}=o(1).
\end{align*}
For  the second inequality we use Lemma \ref{est}
with  $a=\frac{n^2}{l^2},b=(1+o(1))\frac{n^2}{\ell^2},c=t_{2\ell},d=0,  i=k, 
t=\frac{3n\log n}{10r},Q=N$. For the last inequality we used that 
$(\frac{9n^3\log n}{10r\ell^2kN})^k$ in the sum is maximized when $k= \frac{n\log n}{10 r\ell^2 }.$

On the other hand from Lemma \ref{binomial}, it follows that
\begin{align*}
\pr\bigg(N_1\geq \frac{ n \log n}{10r\ell^2 } \bigg) 
\leq \pr\bigg(Binomial\bigg(\frac{n^2}{\ell^2}, \frac{ \log n}{100 r n}\bigg) > \frac{n \log n}{10 r\ell^2 } \bigg) 
 =o(1).
\end{align*}
\end{proof}
\section{Proof of Theorem \ref{Thm2}}
\subsection{Outline of proof}
We first construct a large $\Pi$-colored cycle $C$ containing most of the good vertices. The construction follows the argument from the proof of Theorem \ref{Thm1} and is omitted. The cycle will be such that the vertices not in $C$ can be paired up as $v,v_1$ where $(v,v_1)$ has color $\Pi_k$, say. In addition there will be vertices $v_2,v_3\in C$ such that there is an edge $(v_1,v_2)$ of color $\Pi_{k-1}$ and an edge $(v_1,v_3)$ of color $\Pi_{k+2}$. Then to find a $\Pi$-colored path from a $v$ to $w$ we use $(v,v_1,v_2,Q,w_3,w_1,w)$ where $Q$ is the cycle path starting at $v_2$ in which the indices of the colors ``decrease''.
\subsection{The proof itself}
In this section we will use the following result. Its first inequaltiy follows from the connectivity hitting time result of the random graph process (see \cite{EvolutionGnm}, \cite{KarFr}). 
\begin{lem}\label{conn}
Let $\epsilon>0$. Then w.h.p.\@ $ \frac{(1-\epsilon)n\log n}{2}\leq \tau_1\leq \tau_{\Pi-connected}\leq t_{\Pi}\leq \frac{(1+\epsilon)rn\log n}{2}.$
\end{lem}
The main ingredient of the proof of Theorem \ref{Thm2} is the following Lemma.
\begin{lem}\label{pcycle}
Let  $A\subset GOOD$ be such that $A=O(|BAD|)$. For $i \in [\ell]$ let $\bar{V}_i\subset V_i\cup A$ be such that $\bar{V}_i\subset GOOD$ and $|\bar{V}_i|=\frac{n^2}{\ell^2}-b$ where $b=O(|BAD|)$.
 Then w.h.p.\@ there is a $\Pi$-colored cycle $C=(v_1,v_2,\ldots,v_h,v_1)$ such that 
\begin{enumerate}[(i)]
\item $V(C)=\cup_{i \in \ell} \bar{V}_i$.
\item $E(C)\subseteq \set{ e_i=(a,b): \exists j\in[\ell], a\in \bar{V}_j, b\in \bar{V}_{j+1} \text{ such that } i\leq \tau_{1-\Pi}, c(e_i)=\Pi_j}$.
\end{enumerate}
\end{lem}
Lemma \ref{pcycle} states that if we  remove a set of  $O(|BAD|)$ vertices from each $V_i$ and then relocate $O(|BAD|)$ vertices to form the sets $\bar{V}_i$, given the constraint that the new sets $\bar{V}_i$ are all of the same size,
then w.h.p.\@ there is a $\Pi$-colored cycle $C$ in $G^r_{\tau_{\Pi-connected}}$ that spans $\cup_{i \in \ell} \bar{V}_i$ and ``respects'' the new partitioning $\bar{V_1},...,\bar{V_\ell}$. That is $E(C)$ consists of $\Pi_j$ colored edges from $V_j$ to $V_{j+1}$ that lie in  $G^r_{\tau_{1-\Pi}}$, $j\in [\ell]$.  

The proof of Lemma \ref{pcycle}, which we omit, is similar to the proof of Theorem \ref{Thm1} with the extra advantage that we do not have to take care of any ``bad" vertices. Here we handle vertices in $A$ in the same way that we handled the vertices in $GOOD$ that we shuffled in the proof of Theorem \ref{Thm1}. 
We use the lower bound of $\tau_{\Pi-connected}$ given in Lemma \ref{conn} in place of the lower bound on $\tau_{fit-\Pi}$ given in Corollary \ref{crude}.
\vspace{3mm}
\\In order to prove Theorem \ref{Thm2} we construct a large $\Pi$-colored cycle to which we attach spikes. Here by a spike we mean a 3-star or equivalently 
 the graph on 4 vertices {\color{red}$r_0,r_1,r_2,r_3$ }and edge set $(r_0,r_1), (r_2,r_1)$ and $(r_1,r_3)$. The {\em base} vertices of a spike $r_2,r_3$ will belong to the cycle while typically its {\em head} vertex $r_0$ will be a bad vertex.
\vspace{3mm}
\\Let $h=\min\set{i\in \mathbb{Z}_{\geq 0}: i=-|BAD|\mod \ell}$ and $\Gamma(BAD)=BAD\cup N(BAD)$ where $N(BAD)$ is the neighborhood of $BAD$.
Let $S_B$ be a random subset of $GOOD\setminus \Gamma(BAD)$ of size $h$.
Finally let $B=BAD\cup S_B$. We use $S_B$ to ensure that $\ell$ divides $|B|$.
From Lemma \ref{str} and similar reasoning to that in the proof of Lemma \ref{cover}, for every $v\in B$ we can find a vertex $v_1\in GOOD$ that is adjacent to $v$ in $G_{\tau_{\Pi-connected}}^r$. Let  $\Pi_{k(v)}$ be the color of $(v,v_1)$. For every $v\in B$ we can find  $v_2,v_3\in GOOD$ such that $(v_2,v_1),(v_1,v_3)\in E(G_{2\ell}^r)$ and $(v_2,v_1),(v_1,v_3)$ have colors $\Pi_{k(v)-1}$ and $\Pi_{k(v)+1}$ respectively. In addition the selection of all the vertices above can be done such that all of them are distinct and do not lie in $B$. 

We now construct the sets $\bar{V}_i$, $i\in [\ell]$ as follows.
We begin our construction of $\bar{V}_i$ by first removing from $V_i$ all the vertices in $\set{v,v_1:v\in B}$. Then for every $v\in B$ we move $v_2$ into $\bar{V}_{k(v)-1}$ and $v_3$ into $\bar{V}_{k(v)+2}$. 
After this, for $i\in [\ell]$, we choose a random set $R_i$ of size $4\ell|B|$ from the current vertices in $V_i$, not including the vertices in $\set{v_2,v_3:v\in B}$. 
Finally we redistribute $\cup_{i\in [\ell]} R_i$
such that all the resulting sets $\bar{V}_i$ are of the same size.
By applying Lemma \ref{pcycle} with $A= \set{v_2,v_3:v\in B} \cup(\cup_{i\in [\ell]}R_i)$
 we get a $\Pi$-colored cycle $C$ that spans $\cup_{i \in \ell} \bar{V}_i$ and 
``respects'' the partitioning $\bar{V_1},...,\bar{V_\ell}$.  $C$ along with the edges that belong to the spikes allow us to claim that $G_{\tau_{\Pi-connected}}^r$ is $\Pi$-connected. To see this, we orient the edges of $C$ so that an edge of color $\Pi_i$ is followed by an edge of color $\Pi_{i-1}$, for $i\geq 1$. Then for $v\in B$ we first enter $C$ by travelling along $v,v_1,v_2$. If we wish to travel to $w\in B$ then we travel around $C$ until we reach $w_3$ and finish our path with $w_3,w_1,w$. 
\section{Directed versions}
Observe that now a pattern $\overrightarrow{\Pi}$ has also the notion of direction embedded in it. That is we are looking for an arc of color 
$\overrightarrow{\Pi}_1$ followed by an arc of color $\overrightarrow{\Pi}_2$ that is leaving its out vertex e.t.c.\@
The main difference between the proofs of Theorems \ref{Thm3} and  \ref{Thm4} and the proofs of Theorems \ref{Thm1} and \ref{Thm2} is in defining the demand of a pattern. This is because an in- and an out-arc of the same color are not exchangeable. To deal with this we may think of  the direction of an arc as a second coordinate of its color. The idea is that a vertex ``sees" an in-arc of color red  as an arc of color (red,-) and an out-arc of color blue as an arc of  color (blue,+). Thus instead of looking for a red in-arc and a blue out-arc it looks for arcs of colors (red,-) and (blue,+) respectively. 
\begin{dfn}
Let $\ell \in \mathbb{N}$ then
$$\cD(\overrightarrow{\Pi}):=\set{S\subset [\ell]\times \set{+,-}: \set{(i,+),(i+1,-)} \cap S \neq \emptyset\text{ for all }i\in[\ell]}. $$ 
\end{dfn}
\begin{dfn}
Let $r\in \mathbb{N}$ and let $\overrightarrow{\Pi}$ be a directed $[r]$-pattern. For $i \in [\ell]$ set
$\overrightarrow{\Pi}(i,+)=(\overrightarrow{\Pi}_{i},+)$ and $\overrightarrow{\Pi}(i,-)=(\overrightarrow{\Pi}_{i-1},-)$
  The ``demand" of $\overrightarrow{\Pi}$ is 
$$ d(\overrightarrow{\Pi}):= \min\set{|\{\overrightarrow{\Pi}(i,*):(i,*)\in S\}|:S\in  \cD(\overrightarrow{\Pi})}.$$
\end{dfn} 
Once again for a given  $[r]$-pattern $\overrightarrow{\Pi}$,
if there exists a set $S\in \cD(\overrightarrow{\Pi})$ and a vertex $v\in V$ such that $v$ is not incident to any arc that is assigned one of the at least  $d(\overrightarrow{\Pi})$ (color,sign) elements of  $\{\overrightarrow{\Pi}_i:i\in S\}$ then $v$ does not fit $\overrightarrow{\Pi}$. Conversely if $v$ does not fit $\overrightarrow{\Pi}$ then such a set $S\in\cD(\overrightarrow{\Pi})$ exists.
Each vertex may see $2r$ distinct pairs of (color,direction) among the arcs adjacent to it. 
Thus in place of Lemma \ref{boundssub} we have the following Lemma
\begin{lem}
Let $r=O(1)$, $\overrightarrow{\Pi}$ being a directed $[r]$-pattern and $\epsilon>0$. Then,  w.h.p.\@ 
$$ \frac{2r}{d(\overrightarrow{\Pi})}\cdot\frac{(1-\epsilon) n\log n}{2}  \leq \tau_{fit-\overrightarrow{\Pi}} \leq   \frac{2r}{d(\overrightarrow{\Pi})}\cdot  \frac{(1+\epsilon)n\log n}{2}.$$ 
\end{lem}
For the proofs of Theorems \ref{Thm3} and \ref{Thm4} in the definition of the BAD sets we do not have to impose an ordering on the endpoints of the arcs. Instead we can use the one given by their direction.  
The rest of the proof of Theorem \ref{Thm1} can be extended to the setting of Theorem \ref{Thm3}.
\vspace{3mm}
\\ For the proof of Theorem \ref{Thm4} we can use a similar construction to the one used in the proof Theorem \ref{Thm3}. The idea is once again to construct a $\overrightarrow{\Pi}$-colored cycle $\overrightarrow{C}$ that spans $n-O(|BAD|)$ vertices in $GOOD$. Then we join the rest of the vertices to $\overrightarrow{C}$ by an in- and an out-arc.
As before we can ensure that those arcs exist w.h.p.\@  Finally we can use a subpath of $\overrightarrow{C}$ and the aforementioned arcs to $\overrightarrow{\Pi}$-connect any two vertices.

\end{document}